\documentclass[10pt]{article}
\usepackage{amsmath,amssymb,amsthm}
\usepackage{color}
\usepackage[nohead,margin=1.0in]{geometry}
\usepackage{booktabs}
\usepackage{algorithm2e}

\usepackage{graphicx,caption,subcaption}
\usepackage{url}
\usepackage{mathtools}
\usepackage{amsmath}
\usepackage{amssymb}
\usepackage{commath}
\usepackage{siunitx}
\usepackage[normalem]{ulem}
\usepackage{hyperref}

\theoremstyle{plain}

\newtheorem{remark}{Remark}[section]

\newtheorem{lemma}{Lemma}[section]
\newtheorem{assumption}{Assumption}[section]
\newtheorem{proposition}{Proposition}[section]
\newtheorem{corollary}{Corollary}[section]

\numberwithin{equation}{section}

\newcommand{\R}{\mathbb{R}}

\title{Time-dependent parameter identification in a Fokker-Planck equation based magnetization model of large ensembles of nanoparticles}  
\author{
Hannes Albers\thanks{Center for Industrial Mathematics, University of Bremen,
Bibliothekstr. 5, 28357 Bremen, Germany ({\{halbers, tkluth\}@math.uni-bremen.de})}\and
Tobias Kluth\footnotemark[1]
}

\newcommand{\divg}{\operatorname{div}}
\newcommand{\dx}{\, \textnormal{d}}

\begin{document}
\SetKwComment{Comment}{/* }{ */}

\maketitle

\begin{abstract}

In this article, we consider a model motivated by large ensembles of nanoparticles' magnetization dynamics using the Fokker-Planck equation and analyze the underlying parabolic PDE being defined on a smooth, compact manifold without boundary with respect to time-dependent parameter identification using regularization schemes. In the context of magnetic particle imaging, possible fields of application can be found including calibration procedures improved by time-dependent particle parameters and dynamic tracking of nanoparticle orientation.
This results in reconstructing different parameters of interest, such as the applied magnetic field and the particles' easy axis. These problems are in particular addressed in the accompanied numerical study. 

\end{abstract}

\section{Introduction}
Parameter identification problems for the Fokker-Planck equation in the context of magnetization dynamics of magnetic nanoparticles (MNPs) are of major interest in possible applications in imaging techniques such as magnetic particle imaging (MPI) \cite{Gleich2005}.
In situations where a large number of stochastic processes subject to stochastic ordinary differential equations take place, modeling by a Fokker-Planck (FP) equation can be a helpful tool for mathematical analysis as well as numerical simulation \cite{risken1996fokker}. Instead of considering a large number of stochastic equations to model an ensemble of agents like MNPs, the FP equation is a deterministic partial differential equation that models the probability density function (PDF) of the ensemble's state. Since the FP equation is deterministic, established tools from the analysis of differential equations can be used. In the context of the present work, we consider a time-dependent parabolic partial differential equation (PDE) defined on a smooth manifold. Our goal is to identify time- or time-and-space-dependent parameters in this PDE from noisy measurements.


What results is a nonlinear ill-posed problem which needs to be solved properly by a regularization scheme. For a general basic introduction to this problem we refer to \cite{KaltenbacherNeubauerScherzer+2008}. 
Time-dependent parameter identification problems often have a specific inherent characteristic.
For parameter identification in time-dependent PDEs, the problem has been treated in a very general setting in \cite{Kaltenbacher_2017} and \cite{Nguyen_2019}. In \cite{Kaltenbacher_2018_PI_in_SDEs}, the problem of identifying parameters in a stochastic system via the Fokker-Planck equation is formalized and a method for calculating the adjoint equation is derived. Further works on the topic include \cite{Dunker_2014} for estimation of stationary parameters in Fokker-Planck equations and the recent collection \cite{KaltenbacherSchusterWald_2021} for a number of applications of parameter identification techniques in time-dependent settings.

In parameter identification for PDEs, the goal is to minimize a discrepancy while also identifying the desired parameter, especially in the presence of noise in the measurement. In the field of optimal control, similar problems are considered where the goal is typically to find any admissable control that minimizes a discrepancy term. However, many results from the field of optimal control of Fokker-Planck equations are also highly relevant in a parameter identification setting. For a review of control frameworks for the FP equation in various settings, see \cite{AnnunziatoBorzi2018} and references therein. In particular, in \cite{Aronna2021} a time-dependent control enters the FP linearly and the authors consider first and second-order optimality conditions for the optimal control problem which are derived for very mild regularity requirements on data functions and the spatial domain. In \cite{Fleig2017}, a time and space-dependent control which again enters linearly into the equation is analyzed in a Banach space setting and for a domain with homogenous Dirichlet boundary conditions. 
These findings in the literature influence the present work on 
parameter identification on a domain of a smooth manifold without boundary with time- and space-dependent parameters that enter the equation in specific, not necessarily linear, ways.

The application we focus on is the dynamic behavior of magnetic nanoparticles in the context of MPI. MPI is an emerging imaging modality that exploits magnetic nanoparticles typically suspended in fluid and can be employed for use cases such as imaging blood vessels, tracking medical instruments or magnetic hyperthermia \cite{Knopp2012, Salamon:2016cz, murase2015usefulness}. For modeling the imaging process, the dynamic behavior of MNPs in an applied magnetic field is of central interest \cite{Kluth:2017, Shasha2021}. In prior works, it has been demonstrated that modeling the MNP response in an external field via the Fokker-Planck equation can improve imaging performance compared to simplified models \cite{Albers_toolbox, Kluth_2019, Albers_immobilized}. There are two main principles of dynamic behavior of MNPs suspended in a fluid, namely the physical rotation of the particles (Brownian rotation, \cite{CoffeyKalmykov2012}) and a rotation of the magnetic moment relative to the particle due to micromagnetic effects (N\'{e}el rotation, \cite{Neel1953}). If MNPs are immobilized, the Brownian rotation can be suppressed and pure N\'{e}el relaxation can be observed, which can mathematically be described by the stochastic Landau-Lifshitz-Gilbert equation. If the MNPs are immobilized in the presence of a static magnetic field, their so-called easy axes can be aligned, making the whole ensemble exhibit anisotropic effects. It has been demonstrated \cite{MoeddelGrieseKluthKnopp2021} that this can be exploited to detect the orientation of an ensemble of MNPs. Furthermore, parameter identification on a deterministic version of the Landau-Lifshitz-Gilbert equation has been performed in \cite{Kaltenbacher2021}.

In this work, we introduce the mathematical model and the appropriate function spaces for the Fokker-Planck equation for nanoparticle magnetization dynamics in Section~\ref{sec:param_id}. In Section~\ref{sec:properties}, we prove well-posedness of the forward equation and introduce the parameter identification problem. Then, the existence of a minimum of the discrepancy is established and the Fr\'{e}chet differentiability of the forward operator is proven. Finally, in Section~\ref{sec:problems}, different choices of parameters whose reconstructions are of practical interest in the context of the MPI modality are considered and numerical simulations are presented for the different cases in Section \ref{sec:numerics}.

\section{Parameter identification for the Fokker-Planck equation}\label{sec:param_id}
We consider the Fokker-Planck equation defined on a compact time interval $[0,T]$ and a smooth, compact manifold without boundary $M$. $M$ is assumed to be equipped with a Riemannian metric which induces a norm and an inner product $(\cdot, \cdot)_M$ on the tangent space. In the context of the desired application, we are mainly interested in submanifolds of the Euclidean space, in particular of the 2-sphere embedded in $\mathbb{R}^3$. In this case, the Riemannian metric is induced by the ambient norm in tangent space. The gradient in space $\nabla_M$ is thus also understood as the covariant derivative on $M$, or equivalently the tangential part of the gradient in ambient space \cite{Grinfeld2013-hi}. 
Considering a parameter function $p$ in the drift term, the initial value problem for the Fokker-Planck equation in the general compact manifold case reads\\
\begin{align}
    u' &= \divg_M\left( \lambda\, \nabla_M u  + b(x,t;p) u \right)  &\text{in } M \times [0,T] \label{model_original}\\
    u(\cdot,0) &= u_0 &\text{on } M,\label{initial_value}
\end{align}
where $u' := \frac{\partial u}{\partial t}$ denotes the partial derivative of $u$ with respect to time. The parameter-dependent function $b:M \times [0,T] \to \mathbb{R}^3$ is assumed to belong to a function space $\mathcal{B}$. In order to ensure the well-posedness of the model equation, we choose $\mathcal{B}$ to be a suitable continuously embedded subspace of
\begin{align*}
    \mathcal{\hat{B}} = L^{\infty}\left( 0,T;L^\infty (M; \mathbb{R}^3)\right).
\end{align*}
The choice of the subspace affects the properties of the parameter-to-state operator and is discussed below.
We make the following further assumptions:
\newpage
\begin{assumption}\label{assumption}
\begin{enumerate}
    \item $\lambda > 0$.
    \item $b(x,t;p) = \Gamma (p)(x,t)$ with $\Gamma : \mathcal{P}_\mathrm{ad}\subset \mathcal{P} \to \mathcal{B}$ continuous and bounded, where $\mathcal{P}$ is a Banach space and $\mathcal{P}_\mathrm{ad}$ is a suitable set of admissable parameters
    \item $\mathcal{P}$ is a reflexive Banach space.
\end{enumerate}
\end{assumption}
For the analysis of the model equation, we transition to the weak formulation and consider a Gelfand triple (see e.g. \cite{Evans1998, showalter:1997}) 
\begin{align*}
    V \hookrightarrow H \hookrightarrow V^\ast,
\end{align*}
where $V$ is a separable Banach space, H is a Hilbert space and both embeddings are continuous. Then, the solution space we choose is 
\begin{align*}
    W(0,T) := \{u \in L^2(0,T;V):u' \in L^2(0,T;V^\ast)\}
\end{align*}
Since $W(0,T) \hookrightarrow C([0,T];H)$ \cite{showalter:1997}, the initial value is meaningful. The weak formulation is then obtained by formal multiplication with a test function and then integrating by parts, yielding an equation in the space $L^2(0,T;V^\ast)$:
\begin{lemma}
The weak formulation of the initial value problem \eqref{model_original}-\eqref{initial_value} reads
\begin{align}\label{weak_formulation_complete}
 &- \int_0^T \int_M v'(x,t) u(x,t) \dx x \dx t + \int_0^T \int_M \left[\lambda (\nabla u(x,t), \nabla v(x,t)) + u(x,t)(b(x,t;p),\nabla v(x,t))\right] \dx x \dx t\notag \\
=& \int_0^T \langle f(t), v(t) \rangle_{V'V} \dx t + \int_M u_0(x) v(0,x) \dx x
\end{align}
for  all $ v \in L^2(I,V)$ with $v' \in L^2(I,H)$ and  $v(T)=0$. A right-hand side $f\in L^2(0,T;V^\ast)$ has been added in order to study the more general case. This is equivalent to
\begin{align}
    u' + A(\cdot\,;b)u &= f \quad \text{ in }L^2(0,T;V^\ast)\label{model_weak}\\
    u(0) &= u_0\label{iv_weak}
\end{align}
for the abstract linear operators $A(t;b): V \to V^\ast$ which are defined for almost all $t \in (0,T)$ and all $b \in \mathcal{B}$ according to 
\begin{align*}
    \langle A(t;b) u, v\rangle_{V^\ast,V} = \int_M \left[\lambda (\nabla u(x,t), \nabla v(x,t)) + u(x,t)(b(x,t;p),\nabla v(x,t)) \right]\dx x
\end{align*}
\end{lemma}
\begin{proof}
 Follows directly from \cite[III. Prop. 2.1]{showalter:1997}.
\end{proof}
A natural choice of spaces in this setting is $V = H^1(M)$ and $H = L^2(M)$, which is what we are choosing in the following. 

Now that the governing equation is introduced, we wish to consider the parameter identification problem. For this, we consider the forward operator 
\begin{align*}
    F = G \circ S \circ \Gamma,
\end{align*}
where 
\begin{align*}
    S : \mathcal{B} & \to W(0,T)\\
    b &\mapsto u, \text{ where $u$ solves \eqref{model_weak}-\eqref{iv_weak}}.
\end{align*}
We call $S$ the parameter-to-state or parameter-to-solution map. It maps the drift term $b\in \mathcal{B}$ to the solution $u$ of the model equation with drift term $b$. $b$ can either be considered as the \textit{parameter} of the equation itself, or it can be dependent on a parameter $p$ through the operator $\Gamma$. The observation operator $G: W(0,T) \to 
Y:= L^2(0,T;\mathbb{R}^3)$ is assumed to be linear and continuous; motivated by the application of interest MPI, we consider the expectation of the function $u$, i.e.,
\begin{align}\label{eq:expectation}
    Gu(t) = \int_M m \, u(m,t) \dx m.
\end{align}
As an alternative for illustration purposes and for examining numerical procedures, we will also consider $G$ as the embedding operator, i.e. $G_{\mathrm{id}} := \mathrm{id}_{W(0,T)\to Y}$.

The parameter identification problem can be described as a minimization problem of a cost functional $J: Y \times \mathcal{P} \to \R$ in general. The problem can then be formulated as obtaining the optimal parameter $p^\ast$ via
\begin{align}
    p^\ast \in  \underset{p\in \mathcal{P}_\mathrm{ad}}{\operatorname{argmin}} \,J(F(p),p),\label{eq:optimization}
\end{align}
where the cost functional $J$ depends on a (typically noisy) observation $y^\delta$ and is weakly lower semicontinuous with respect to the input tuple. A typical choice for $J$ is the $L^2$-discrepancy
\begin{align*}
    J_{L^2}(F(p),p) &:= \|F(p) - y^\delta\|_{L^2(0,T;\mathbb{R}^3)}^2,
\end{align*}
which fulfills the requirement and which we will consider together with the nonlinear Landweber method as a regularization scheme \cite{HankeNeubauerScherzer95}. Other standard choices include the Tikhonov regularized $L^2$-discrepancy as treated in \cite{Fleig2017} among others like, e.g., sparsity regularization \cite{Jin_2012} or $L^p$-discrepancies for $p\neq 2$ \cite{SchusterKaltenbacherHofmannKazimierski+2012}.

\subsection{Properties of the forward operator}\label{sec:properties}
In order to analyze the parameter identification problem
\begin{align*}
    \text{find }p \in \mathcal{P}_\mathrm{ad}:\quad  F(p) = y,
\end{align*}
for given noise-corrupted measurement $y^\delta \in Y = L^2(0,T;\mathbb{R}^3)$ with $ \|y-y^\delta\|\leq \delta$, we first need to establish well-definedness of the forward operator $F$. Since $G$ and $\Gamma$ are well-defined by definition, well-definedness of $S$ remains to be shown, i.e., the fact that the initial value problem \eqref{model_weak}-\eqref{iv_weak} admits a unique solution for all $b \in \mathcal{B}$.
\begin{lemma}
\label{continuity_coercivity}
With Assumption \ref{assumption}, the operator $A$ as in \eqref{model_weak}-\eqref{iv_weak} satisfies for almost all $t \in (0,T)$, all $b \in \mathcal{B}$ and all $u,v \in V$:
\begin{align}
\langle A(t;b)u,v\rangle_{V^\ast,V} &\leq C\, \|u\|_V \|v\|_V, \\
\langle A(t;b)u, u \rangle_{V^\ast,V} &\geq \frac{\lambda}{2}\norm{u}^2_{H^1(M)} - Q\norm{u}^2_{L^2(M)},\label{eq:garding}
\end{align}
with $C > 0$,  $Q=\frac{\lambda}{2} + \frac{1}{2\lambda} \|b\|_\mathcal{\hat{B}}^2$.
\end{lemma}
\begin{proof}
For the first part, apply Hölder's inequality. For the second part, use Young's inequality to obtain for any $\epsilon > 0$
\begin{equation*}
 \left|\int_M u(b,\nabla u) \dx x \right| \leq \int_M |u| |b| |\nabla u| \dx x \leq \int_M \epsilon |\nabla u|^2 + \frac{1}{4\epsilon} |u|^2 |b|^2 \dx x \leq \epsilon \| \nabla u \|_{L^2(M)}^2 + \frac{1}{4\epsilon} \| u \|_{L^2(M)}^2 \|b\|_{L^\infty((0,T)\times M)}^2.
\end{equation*}
 Using this, we can estimate 
\begin{align*}
 \langle A(t)u, u \rangle_{V'V} &\geq \lambda \|\nabla u\|_{L^2(M)}^2 + \int_M u (b,\nabla u)_M \dx x \\
& \geq \lambda \|\nabla u\|_{L^2(M)}^2 - \epsilon \| \nabla u \|_{L^2(M)}^2 - \frac{1}{4\epsilon} \| u \|_{L^2(M)}^2 \|b\|_{L^\infty((0,T)\times M)}^2 \\
&=(\lambda-\epsilon) \|u\|_{H^1(M)}^2 - \left[(\lambda -\epsilon)  + \frac{1}{4\epsilon}\|b\|_{L^\infty((0,T)\times M)}^2 \right] \| u \|_{L^2(M)}^2  
\end{align*}
Choosing $\epsilon=\lambda/2$, we obtain \eqref{eq:garding}.

\end{proof}

With these estimates, we can immediately evoke standard results from the literature to prove that $S$ is indeed well-defined.

\begin{proposition}\label{existence_uniqueness}
With Assumption \ref{assumption}, the initial value problem \eqref{model_weak}-\eqref{iv_weak} has a unique solution $u \in W(0,T)$ for each $b \in \mathcal{B}$. The solution is nonnegative almost everywhere if the initial value is, and it satisfies
\begin{align*}
    \int_M u(t,x) \, dx &= \int_M u_0(x) \, dx \quad \text{for all }t \in [0,T].
\end{align*}

\end{proposition}
\begin{proof}
See \cite[Thm. 11.7]{chipot2012elements} for existence and uniqueness. The assumptions therein about continuity and coercivity of the operator $A$ are proven in Lemma \ref{continuity_coercivity}.
For the nonnegativity, see the weak maximum principle in \cite[Thm. 11.9]{chipot2012elements}. The proof works without modification in the manifold case. For the third claim, we can observe without difficulty:
\begin{align*}
    \frac{\partial }{\partial t} \int_M u \, dx &= \int_M -\divg_M \left( \lambda \nabla u + b u \right) \, dx,\\
    &= 0
\end{align*}
by applying the divergence theorem to $M$ (whose boundary is empty). Thus, the mean value of $u$ remains constant over time. 
\end{proof}
\begin{remark}
In particular, if $u_0$ is a probability density defined on $M$, then the solution $u(t,\cdot)$ of \eqref{model_weak}-\eqref{iv_weak} will be a probability density for each $t \in [0,T]$.
\end{remark}
We also obtain an a-priori estimate for the solution $u$:
\begin{proposition}\label{estimates}
In the setting of Proposition \ref{existence_uniqueness}, the following estimates hold:
\begin{align*}
    \|u\|^2_{L^\infty(0,T;H)} &\leq C_1 \left( \|u_0\|^2_H + \|f\|^2_{L^2(0,T;V^\ast)}\right),\\
    \|u\|_{L^2(0,T;V)}^2 &\leq C_2 \left( \|u_0\|^2_H + \|f\|^2_{L^2(0,T;V^\ast)}\right),\\
    \|u'\|^2_{L^2(0,T;V^\ast)} & \leq C_3 \left( \|u_0\|^2_H + \|f\|^2_{L^2(0,T;V^\ast)} \right).
\end{align*}
with $C_1,C_2,C_3>0$ and all depending continuously on $\|b\|_{\hat{\mathcal{B}}}$. In particular, a $C>0$ exists such that 
\begin{align}\label{estimate_W}
    \|u\|^2_{W(0,T)} &\leq C \left( \|u_0\|^2_H + \|f\|^2_{L^2(0,T;V^\ast)} \right).
\end{align}
\end{proposition}
\begin{proof}
Applying $u(t)$ to equation \eqref{model_weak} for $t\in (0,T)$ yields, using Lemma \ref{continuity_coercivity}:
\begin{align}
    \frac{1}{2} \frac{d}{dt} \|u(t)\|^2_H + \langle A(t)u(t),u(t)\rangle_{V^\ast,V} &= \langle f(t),u(t) \rangle _{V^\ast,V} \notag\\
    \implies \frac{d}{dt} \|u(t)\|^2_H + \lambda \|u(t)\|^2_V &\leq 2 \langle f(t),u(t) \rangle + 2Q(t) \|u(t)\|^2_H \notag\\
    &\leq 2\epsilon \|u(t)\|_V^2 + \frac{1}{2\epsilon}\|f(t)\|^2_{V^\ast} + 2Q(t) \|u(t)\|^2_H.\label{young}
\end{align}
With $\varepsilon = \frac{\lambda}{2}$, we get
\begin{align*}
    \frac{d}{dt}\|u(t)\|^2_H + \lambda \|u(t)\|^2_V &\leq \lambda \|u(t)\|_V^2 + \frac{1}{\lambda} \|f(t)\|^2_{V^\ast} + 2Q(t) \|u(t)\|^2_H.
\end{align*}
We apply the Gronwall inequality.
\begin{align*}
\|u(t)\|_H^2 &\leq e^{\int_0^t 2Q(\tau)\, d\tau } \left( \|u(0)\|^2_H + \frac{1}{\lambda}\int_0^t \|f(\tau)\|^2_{V^\ast}\, d \tau \right).
\end{align*}
This yields the first estimate. If we choose $\epsilon = \lambda/4$ in \eqref{young} and integrate from $0$ to $T$, we obtain the second estimate. The third estimate can be obtained by applying an arbitrary $\phi \in L^2(0,T;V)$ to \eqref{model_weak} and inserting the first two estimates.
\end{proof}

Next, we prove that the optimization problem formulated above has at least one solution. For this, we follow the general structure of the proofs in \cite[Thm. 3.1]{Aronna2021} and \cite[Thm. 5.2]{Fleig2017}. First, we need some basic results about compact embeddings of Sobolev and Bochner spaces.
\begin{lemma}
If $M$ is a compact manifold, $s \in \mathbb{R}$, then the embeddings
\begin{align*}
    H^{s+\sigma} (M) \hookrightarrow H^s(M)
\end{align*}
are compact for $\sigma> 0$. In particular,
\begin{align*}
    H^1(M) \hookrightarrow L^2(M) \text{ compactly.}
\end{align*}
\end{lemma}
\begin{proof}
See \cite{taylor2010partial}, Chapter 4, Proposition 3.4.
\end{proof}
\begin{lemma}[Aubin-Lions]\label{aubin-lions}
If $B_0, B, B_1$ are Banach spaces with $B_0 \hookrightarrow B \hookrightarrow B_1$ continuous embeddings with the first one being compact, $1 < p,q, < \infty$, $0 < T < \infty$, and $B_0$ and $B_1$ reflexive, then the Banach space
\begin{align*}
    W:= \{u \in L^p(0,T;B_0): u'\in L^q(0,T;B_1)\}
\end{align*}
is compactly embedded into $L^p(0,T;B)$.
\end{lemma}
\begin{proof}
See e.g. \cite{roubicek2006nonlinear}, Lemma 7.7.
\end{proof}
\begin{proposition}\label{opt_weakly_closed}
If Assumption \ref{assumption} holds, then the parameter-to-state operator $S: \mathcal{D}(S) \subset \mathcal{B} \to W(0,T)$ is weakly sequentially closed for any weakly closed domain $\mathcal{D}(S)$, i.e., for any sequence $(b_k) \subset \mathcal{D}(S)$ we have
\begin{align*}
    b_k \rightharpoonup \bar{b} \ \land  \ S(b_k) \rightharpoonup \bar{u} \text{ in }W(0,T) \implies \bar{b} \in \mathcal{D}(S) \ \land  \  S(\bar{b}) = \bar{u},
\end{align*}
where we call a set weakly closed iff it is closed w.r.t. the weak topology. In particular, (strongly) closed and convex sets fulfill that requirement \cite{yosida78}.
\end{proposition}
\begin{proof}
Let $(b_k) \subset \mathcal{D}(S) \subset \mathcal{B}$ be a weakly convergent sequence to $\bar{b}\in \mathcal{B}$ and $(S(b_k))=: (u_k)$ weakly convergent to $\bar{u}\in W(0,T)$. We observe that $\bar{b} \in \mathcal{D}(S)$ since the domain is weakly closed. We want to prove that $S(\bar{b}) = \bar{u}$.
With Lemma \ref{aubin-lions}, we obtain the compact inclusion $W(0,T) \hookrightarrow L^2((0,T)\times M)$. Passing to a subsequence, we can thus conclude
\begin{align*}
    u_k &\rightharpoonup \bar{u} \quad \text{weakly in }L^2(0,T;H^1(M)),\\
    u_k' &\rightharpoonup \bar{u}' \quad \text{weakly in }L^2(0,T;H^1(M)^\ast),\\
    u_k &\to \bar{u} \quad \text{strongly in }L^2((0,T)\times M).
\end{align*}
This yields directly the convergence of two of the terms in the weak formulation of the state equation \eqref{weak_formulation_complete} -- for $v \in L^2(0,T;H^1(M))$ with $v' \in L^2((0,T)\times M)$:
\begin{align*}
    \int_0^T \int_M v' u_k \dx x \dx t & \to \int_0^T \int_M v' \bar{u} \dx x \dx t,\\
    \int_0^T \int_M \lambda \nabla u_k \cdot \nabla v \dx x \dx t &\to \int_0^T \int_M \lambda \nabla \bar{u} \cdot \nabla v \dx x \dx t.
\end{align*}
What remains is to check the convergence of $\int_0^T \int_M u_k \, b_k \cdot \nabla v \dx x \dx t$.
We have $u_k \nabla v \to \bar{u}\nabla v$ strongly in $L^1((0,T)\times M;\mathbb{R}^3)$. Since $\mathcal{B}\hookrightarrow \mathcal{\hat{B}} = L^\infty(0,T;L^\infty(M))$ continuously, $(b_k)$ converges weakly also in $\hat{\mathcal{B}}$, and thus it converges also weak-$\ast$ in the latter.  Since a weakly-$\ast$-convergent sequence of functionals applied to a strongly convergent sequence is convergent, we conclude
\begin{align*}
    \int_0^T \int_M u_k \, b_k \cdot \nabla v \dx x \dx t \to \int_0^T \int_M \bar{u} \, \bar{b} \cdot \nabla v \dx x \dx t
\end{align*}
due to the weak-$\ast$-convergence of $(b_k)$ in $\hat{\mathcal{B}}$. Thus, $S(\bar{b}) = \bar{u}$ and the claim holds.
\end{proof}
\begin{corollary}
If $\mathcal{D}(F)$ is a weakly closed set, the forward operator $F$ is weakly closed in the setting of Prop. \ref{opt_weakly_closed}.
\end{corollary}
\begin{proof}
    Since $\Gamma$ is weakly continuous by Assumption \ref{assumption} and $G$ is linear and bounded and thus also weakly continuous, the claim follows.
\end{proof}
\begin{remark}
    If $F$ is a mapping between separable Hilbert spaces, we observe that $F$ is weakly closed, continuous and compact since $G$ is compact. Then, we can invoke \cite[Prop. A3]{Engl_Kunisch_Neubauer} and conclude that the inverse problem for $F$ is locally ill-posed in $\operatorname{int}(\mathcal{D}(F))$. This makes employing regularization techniques necessary for reconstruction.
\end{remark}

\begin{proposition}
In the setting of Prop. \ref{opt_weakly_closed}, the minimization problem \eqref{eq:optimization} admits a solution $p \in \mathcal{P}_\mathrm{ad}$, assuming that $\mathcal{P}_\mathrm{ad}$ is bounded and weakly closed  (e.g., $\mathcal{P}_\mathrm{ad} = \overline{B_r(\hat{p})}$ for some $r>0$ and $\hat{p} \in \mathcal{P}$) provided that $J: L^2(0,T;\mathbb{R}^3) \times \mathcal{P} \to \mathbb{R}_+$ is weakly lower semicontinuous with respect to the input tuple.
\end{proposition}
\begin{proof}
Since $J$ is bounded from below by 0, there exists a minimizing sequence $(y_k, p_k) = (F(p_k),p_k) \subset L^2(0,T;\mathbb{R}^3) \times \mathcal{P}_\mathrm{ad}$ such that 
\begin{align*}
    J(y_k, p_k) \to \inf_{p \in \mathcal{P}_\mathrm{ad}} J(F(p),p).
\end{align*}
Since $\mathcal{P}_\mathrm{ad}$ is bounded, so is $(p_k)$, so it admits a weakly convergent subsequence $p_k \rightharpoonup \bar{p} \in \mathcal{P}$. Since boundedness of $(p_k)$ in $\mathcal{P}$ implies boundedness of $(b_k) = (\Gamma(p_k))$ in $\mathcal{B}$ due to the boundedness of $\Gamma$, estimate \eqref{estimate_W} ensures boundedness of $(y_k)$ in $L^2(0,T;\mathbb{R}^3)$, i.e., we get by extracting a subsequence
\begin{align*}
    F(p_k) = y_k \rightharpoonup \bar{y} \in L^2(0,T;\mathbb{R}^3).
\end{align*}
Since $J$ is assumed to be weakly lower semicontinuous, we obtain
\begin{align*}
    J(\bar{y},\bar{p}) \leq \liminf_{k \to \infty} J(y_k,p_k).
\end{align*}
Weak sequential closedness of $F$ ensures that $\bar{p} \in \mathcal{P}_\mathrm{ad}$ and $F(\bar{p}) = \bar{y}$, proving the claim.
\end{proof}

Next, we employ the a-priori-estimates proven in Prop. \ref{estimates} in order to prove the differentiability of the parameter-to-state map $S$.

\begin{proposition}
The parameter-to-state operator $S$ is differentiable with respect to the topology of $L^2(0,T;H^2(M))$ in the sense that for $b, b + h \in \Lambda \subset \mathcal{B}$ for a bounded set $\Lambda$, it holds
\begin{align*}
    \frac{\| S(b+h) - S(b) - S'(b)h \|_{W(0,T)}}{\| h \|_{L^2(0,T;H^2(M))}} \to 0 \quad \text{as } \|h\|_{L^2(0,T;H^2(M))} \to 0,
\end{align*}
where $S'(b)h =: v$ solves
\begin{align*}
    v' + A(b)v &= -\divg_M(S(b) h)\\
    v(0) = 0.
\end{align*}
Here, the $\divg$ operator is to be understood in the distributional sense.
\end{proposition}
\begin{proof}
First, the a priori estimates in Prop. \ref{estimates} yield
\begin{align*}
    \|v\|^2_{W(0,T)} \leq C \| \divg_M(S(b)h)\|^2_{L^2(0,T;V')}.
\end{align*}
The constant $C$ depends on the $L^\infty$-norm of $b$. Since $b$ is contained in a bounded subset of $\mathcal{B}$ which is continuously embedded in $\mathcal{\hat{B}}$, $C$ can be taken independently of $b$.

Since, for $\xi \in L^2(0,T;V)$, we have
\begin{align*}
    \langle \divg_M(S(b)h), \xi \rangle_{L^2(0,T;V'), L^2(0,T;V)} &= \int_0^T \int_M S(b) h \cdot \nabla \xi \dx x \dx t\\
    &\leq \| S(b) h\|_{L^2((0,T)\times M)} \| \xi \|_{L^2(0,T;V)},
\end{align*}
we can estimate
\begin{align*}
    \| \divg_M(S(b) h) \|^2_{L^2(0,T;V')} & \leq \| S(b) h \|^2_{L^2((0,T)\times M)}\\
    &\leq \|S(b)\|^2_{L^\infty(0,T;H)} \, \| h \|^2_{L^{2}(0,T;L^\infty(M))},\\
    &\leq \|S(b)\|^2_{W(0,T)} \|h\|^2_{L^{2}(0,T;L^\infty(M))}.
\end{align*}
since $W(0,T) \hookrightarrow L^\infty(0,T;H)$.

Consider $\hat{u}_h := S(b+h)-S(b)$. This function solves
\begin{align*}
    \hat{u}_h + A(b)\hat{u}_h &= -\divg(S(b+h)h)\\
    \hat{u}_h(0) &= 0.
\end{align*}
The a-priori estimate yields
\begin{align*}
    \|\hat{u}_h\|^2_{W(0,T)} &\leq C(b) \|S(b+h)\|^2_{W(0,T)} \|h\|^2_{L^{2}(0,T;L^\infty(M))}\\
    &\leq C(b, b+h) \|u_0\|^2_H \|h\|^2_{L^{2}(0,T;L^\infty(M))}.
\end{align*}
Lastly, $w = \hat{u}_h - v $ solves 
\begin{align*}
    w' + A(b) w &= -\divg (\hat{u}_h h),\\
    w(0) = 0;
\end{align*}
With this, we can estimate
\begin{align*}
    \|w\|^2_{W(0,T)} &\leq C(b,b+h) \|u_0\|^2_H  \|h\|^4_{L^{2}(0,T;L^\infty(M))}.
\end{align*}
Since $b, b+h \in \Lambda$, $C$ can be bounded uniformly in $(b, b+h)$ and thus we obtain
\begin{align*}
    \frac{\|S(b+h)-S(b)-v\|_{W(0,T)}}{\|h\|_{L^{2}(0,T;L^\infty(M))}} \to 0 \quad \text{ as } \|h\|_{L^{2}(0,T;L^\infty(M))} \to 0.
\end{align*}
Since $H^2(M) \hookrightarrow L^\infty(M)$, we obtain the same result for $h \in \Lambda \cap L^2(0,T;H^2(M))$.
\end{proof}
\begin{corollary}
    As shown in the final part of the proof, the above result holds also with respect to the topology of the space $L^2(0,T;L^\infty(M))$.
\end{corollary}
\begin{remark}
Since $\hat{\mathcal{B}} = L^\infty(0,T;L^\infty(M)) \hookrightarrow L^2(0,T;L^\infty(M))$, we conclude that $S$ is a Fréchet differentiable operator on $\mathcal{B}$ due to its continuous embedding into $\mathcal{\hat{B}}$.
\end{remark}
\begin{remark}
If $\Gamma : \mathcal{P} \to \mathcal{B}$ is Fréchet differentiable, we obtain that 
\begin{align*}
    G \circ S \circ \Gamma = F: \mathcal{P} \to Y = L^2(0,T;\mathbb{R}^3)
\end{align*}
is Fréchet differentiable due to the chain rule and the linearity and boundedness of $G$.
\end{remark}
In order to apply the Landweber method to the forward operator $F$, we require the Hilbert space adjoint to its Fréchet derivative $F'$ \cite{Kaltenbacher_2017}.
\begin{proposition}\label{hilbert_space_adjoint}
The Hilbert space adjoint of the derivative of $F = G\circ S\circ \Gamma$ at $p\in \mathcal{P}$, $F'(p)^\star: L^2(0,T;\mathbb{R}^3) \to \mathcal{P}$, is given as 
\begin{align*}
    F'(p)^\star z &= E_{\mathcal{P}}^{-1} \Gamma'(p)^\ast S(p) \nabla \psi^z,
\end{align*}
where $\Gamma'(p)^\ast$ is the Banach space adjoint of $\Gamma'(p)$, and $\psi^z$ solves the adjoint equation
\begin{equation}
\begin{aligned}\label{adj_eq}
    -\frac{\partial}{\partial t}\psi^z &= \lambda \Delta \psi^z + b \cdot \nabla \psi^z + G^\ast z,\\
    \psi^z(T) &= 0.
\end{aligned}
\end{equation}
Here, $E_{\mathcal{P}} : \mathcal{P} \to \mathcal{P}^\ast$ is the Riesz isomorphism between $\mathcal{P}$ and $\mathcal{P}^\ast$.
\end{proposition}
\begin{proof}
See \cite{Nguyen_2019} and the references therein, applied to the composition of $S$ with $\Gamma$. Here, we do not identify $\mathcal{P}$ with $\mathcal{P}^\ast$ due to the different inner products, which is why we require the Riesz isomorphism. 
\end{proof}
\begin{remark}
If we choose $\mathcal{P} =  H^1(0,T;H^2(M))$, we obtain for the Riesz isomorphism
$E_\mathcal{P}^{-1}: \mathcal{P}^\ast \to \mathcal{P}$ that $u = E^{-1}_\mathcal{P} f$ can be computed by solving two simple linear equations:
\begin{align*}
    v &\in \mathcal{P}: \quad v - v'' = f, \quad v'(0) = v'(T) = 0 \quad \text{and }\\
    u &\in \mathcal{P}: \quad u - \Delta u + \Delta^2 u = v.
\end{align*}
For $\mathcal{P} = L^2(0,T;H^2(M))$, only the second equation has to be solved for $v=f$. For $\mathcal{P} = H^1(0,T, \mathbb{R}^3)$, only the first equation has to be solved and $u = v$.
\end{remark}
We note that, in practice, the strength of the smoothing in time for the Riesz isomorphism can be controlled by rescaling time or, equivalently, equipping the space $H^1(0,T; H^2(M))$ with the equivalent norm
\begin{align*}
    \norm{u}_{H^1_\varepsilon(0,T;H^2(M))} := \norm{u}_{H^2(M)} + \varepsilon\norm{u'}_{H^2(M)}
\end{align*}
for $\varepsilon > 0$, and analogously for the smoothing in space.

\section{Selected problems for nanoparticle dynamics in MPI}\label{sec:problems}
In order to apply the theory to the Fokker-Planck equation for the dynamic behavior of an ensemble of magnetic nanoparticles in an applied magnetic field, we define the general form of the Fokker-Planck equation for an ensemble of particles subjected to N\'{e}el relaxation. Here, the independent variable $m$ (denoted as $x$ for the general case in the previous section) denotes the possible directions of a nanoparticle's magnetic moment vector. Thus, the manifold $M$ is always chosen to be the two-dimensional unit sphere $S^2$, which is a smooth, compact Riemannian manifold without boundary such that the previously derived theory applies. The drift term $b$ in equation \eqref{model_original} then reads \cite{Kluth:2017}:
\begin{align*}
    b(m,t) &= \alpha_1 (m \times H_\mathrm{app}) \times m + \alpha_2 (m \times \varphi (m,t) ) \times m,
\end{align*}
for $\alpha_1, \alpha_2>0 $ physical constants, $H_\mathrm{app}:[0,T] \to \mathbb{R}^3$ the applied magnetic field and $\varphi: S^2 \times [0,T] \to \mathbb{R}^3$ is what we call the anisotropy landscape of the nanoparticles. More precisely, $\varphi$ is the derivative of the anisotropy energy of a particle with respect to the magnetic moment direction, excluding constant factors which are assumed to be given as $\alpha_2$. Assuming uniaxial anisotropy, we have $\varphi(m,t) = (m^Tn)n $ with $n = n(t) \in S^2$ the anisotropy direction, and $\alpha_2 = 2\tilde{\gamma}\hat{\alpha}\frac{K^\mathrm{anis}}{M_S}$, where $\hat{\alpha}$ is a damping factor, $\tilde{\gamma}:= \gamma/(1+\hat{\alpha}^2)$ is the modified gyromagnetic ratio with the gyromagnetic ratio $\gamma$, $M_S$ is the saturation magnetization and $K^\mathrm{anis}$ is the uniaxial anisotropy constant. The other physical constant is given as $\alpha_1 = \tilde{\gamma}\hat{\alpha}\mu_0$, where $\mu_0$ is the vacuum permeability \cite{Albers_toolbox}. For further reading on the particle physics, see \cite{binns_tutorial, Neel1953, stoner1991mechanism}.

In this context, we have in particular three possible parameters $p$ that we may want to identify from measurements:
\begin{enumerate}
    \item Identify time-dependent $p:= H_\mathrm{app}$; then $\mathcal{P} = H^1(0,T;\mathbb{R}^3) \hookrightarrow L^\infty(0,T;\mathbb{R}^3)$ and $\mathcal{B} = \mathcal{\hat{B}}$.
    \item Identify space-time dependent $p:= \varphi$, with $\mathcal{P} = H^1(0,T;H^2(S^2;\mathbb{R}^3)) \hookrightarrow L^\infty(0,T;L^\infty(S^2;\mathbb{R}^3))$ and $\mathcal{B} = \hat{\mathcal{B}}$. 
    \item Assume uniaxial anisotropy, i.e., $\varphi(m,t) = (m^T n(t)\,)\,n(t)$, and identify time-dependent $p:= n$, where $\mathcal{P} = H^1(0,T;\mathbb{R}^3) \hookrightarrow L^\infty(0,T;\mathbb{R}^3)$ and $\mathcal{B} = \hat{\mathcal{B}}$.
\end{enumerate}
\begin{lemma}
    For each of the choices of $(p, \mathcal{P})$ defined above, the resulting forward operator $F: \mathcal{P} \to Y$ is Fréchet differentiable.
\end{lemma}
\begin{proof}
    In case 1, since $\Gamma$ is affine in $p$, it is an affine Nemytskij operator induced by a differentiable, bounded function and it is thus bounded, continuous and Fréchet differentiable, see e.g. \cite{troeltzsch2010optimal, goldberg1992}. The same holds for case 2.

    In the third case, $\Gamma$ is a nonlinear Nemytskij operator between Lebesgue-Bochner spaces. We consider the inducing function $\hat{\Gamma}: (0,T) \times \mathbb{R}^3 \to L^\infty(S^2; \mathbb{R}^3)$ with
    \begin{align*}
        \hat{\Gamma}(t,n)(m) := \alpha_1 (m \times H_\mathrm{app}(t)) \times m + \alpha_2 (m \times (m^Tn)n ) \times m.
    \end{align*}
    This function is Fréchet differentiable with respect to $n$ and $\frac{\partial}{\partial n}\hat{\Gamma}(t,n) = \alpha_2 (m \times (nm^T+n^Tm\, \mathrm{Id}_3) \times m$ is linear in $n$ and thus a Carathéodory function. Since it is also bounded uniformly in time if $H_\mathrm{app} \in L^\infty(0,T;\mathbb{R}^3)$, we can invoke \cite[Thm. 7]{goldberg1992} for $p=q=\infty, X = \mathbb{R}^3, Y=L^\infty(S^2;\mathbb{R}^3)$, in the nomenclature of the cited theorem, and conclude that the Nemytskij operator generated by $\hat{\Gamma}$, namely $\Gamma(n)(t,m) := \hat{\Gamma}(t,n)(m)$ is Fr\'{e}chet differentiable from $L^\infty(0,T;\mathbb{R}^3)$ into $L^\infty(0,T;L^\infty(S^2;\mathbb{R}^3))$.

    Due to the chain rule, we can conclude that in all three cases, the forward operator $F: \mathcal{P} \to Y$ for $Y = L^2(0,T;\mathbb{R}^3)$ is Fr\'{e}chet differentiable.
\end{proof}

\section{Numerical Experiments}\label{sec:numerics}

The different parameter identification scenarios introduced in the preceding section are examined numerically in the following. For this, we consider the inverse problem
\begin{align*}
    \text{Find } p\in \mathcal{P}: F(p) &= y,
\end{align*}
where $Y \ni y^\delta = y + \eta$ is a given (simulated) measurement corrupted with a noise component $\eta$ with $\|\eta \|_Y \leq \delta$. Finding an optimal parameter $p^\ast$ will be realized by iteratively minimizing the $L^2$-discrepancy, i.e., we consider the discrepancy functional
\begin{align*}
    J(y^\delta, p) &= \| F(p) - y^\delta \|_Y^2
\end{align*}
for $Y = L^2(0,T;\mathbb{R}^3)$. Since regularization is necessary in general, in particular in the presence of noise, we employ the nonlinear Landweber method, which, together with an appropriate stopping criterion, constitutes a regularization method \cite{HankeNeubauerScherzer95}. Starting with an initial guess $p_0$, the iterates will be computed via
\begin{align*}
    p_{k+1} &= p_k - \omega_k \, (F'(p_k)^\star (F(p_k) - y^\delta),
\end{align*}
where the step length $\omega_k$ (at least approximately) fulfills $\omega_k \leq \|F'(p_k)^\star\|_\mathrm{op}$. The Hilbert space adjoint of $F'(p_k)$ is computed as presented in Prop. \ref{hilbert_space_adjoint}. For a full evaluation, we will conduct the Landweber iteration to convergence of the discrepancy, storing the iterates, and then evaluate the results using the discrepancy principle as well as the best achievable iterate compared to the ground truth. 

\subsection{Algorithmic considerations}
For the different settings described in Section \ref{sec:problems}, the forwards and backwards equations are implemented using a finite volume method, see \cite{Albers_toolbox}.
To avoid inverse crime, the ground truth is computed on a fine grid (20480 triangles on the unit sphere) and then interpolated to a coarser grid (5120 triangles), on which the algorithm is then applied. The triangle circumcenters of the coarser grid are pairwise distinct from those of the finer one.
After interpolation to the coarser grid, the observation operator is applied and Gaussian noise is added, where the intensity is chosen relative to the $L^\infty$-norm of the synthetic observation. 
For the step length, an upper bound is estimated numerically for each setting by approximating $\norm{F'(\tilde{p})}_\mathrm{op}$ for some realistic values of $\tilde{p}$. Armijo backtracking is used with a reduction factor of $0.7$ for the step length to ensure a monotonous decrease of the discrepancy.
Also, since the identification of the parameters is only possible on the interval $(0,T)$, the evaluation of the reconstruction error is considered on $(0,T)$.
The computational steps for the parameter identification are  outlined in Algorithm \ref{algo}. Since computational issues arise during the first iterations when estimating a time-dependent $n$, it is necessary to find a suitable initial value instead of the arbitrary constant value we use otherwise. Thus, in this case, we employ the algorithm without the smoothing step to find a better initial value first. This is controlled by enabling the flag $\mathrm{FIND\_INITIAL\_VALUE}$. The implementation used to generate the results shown can be found under \url{https://gitlab.informatik.uni-bremen.de/mpi-ip/pi-for-mnps}.

\begin{algorithm}
\SetKwData{Left}{left}\SetKwData{This}{this}\SetKwData{Up}{up}
\SetKwFunction{Union}{Union}\SetKwFunction{FindCompress}{FindCompress}
\SetKwInOut{Input}{input}\SetKwInOut{Output}{output}

\Input{Measurement $y^\delta$, default step length $\omega$, Armijo reduction factor $\alpha_A$, maximum number of Armijo steps $j_\mathrm{max}$, convergence tolerance $TOL$, maximum number of iterations $k_\mathrm{max}$, initial guess $p_1$, flag $\mathrm{FIND\_INITIAL\_VALUE}$}\
\Output{$p_k, y_k$}
\BlankLine
$y_1 \gets F(p_1)$\;
$\mathrm{disc}_1 = ||y_1 - y^\delta||$ \;
\For{$k = 2,\dots,k_\mathrm{max}$}{
$\mathrm{res} \gets G^\ast (y^\delta - y_{k-1})$\;
$\Psi \gets$ solution of adj. equation \eqref{adj_eq} for $p = p_{k-1}$\;
$\mathrm{grad} \gets \Gamma'(p_{k-1})^\ast S(p_{k-1}) \nabla \Psi$\;
\For(\tcp*[f]{Inner loop to determine the step size}){$j = 1,\dots, j_\mathrm{max}$}
{
$p_\mathrm{tmp} \gets p_{k-1} - \alpha_A^{j-1} \,\omega  \,\mathrm{grad}$\;
\If{$\neg \mathrm{FIND\_INITIAL\_VALUE}$}{$p_\mathrm{tmp} \gets \mathrm{smooth}(p_\mathrm{tmp})$\;}
$y_\mathrm{tmp} \gets F(p_\mathrm{tmp})$\;
$\mathrm{disc}_k  \gets || y_\mathrm{tmp} - y^\delta ||$\;
\If{$\frac{\mathrm{disc}_{k-1}-\mathrm{disc}_k}{\mathrm{disc}_{k-1}} > TOL$}{
$p_k \gets p_\mathrm{tmp}$\;
$y_k \gets y_\mathrm{tmp}$\;
break\;
}

}

}
\caption{Algorithm for iteratively identifying an unknown parameter $p$ from a given measurement $y^\delta$. If the measurement contains noise, the iterates $p_k$ must be stored or a different stopping criterion has to be checked at each iteration in order to avoid overfitting to data. If the parameter $p$ that is to be identified is the particles' easy axis $n$, then the algorithm is first run with the flag $\mathrm{FIND\_INITIAL\_VALUE}$ enabled in order to to compute a suitable initial value.}\label{algo}
\end{algorithm}

\subsection{Numerical results}
In order to test the reconstruction of the applied magnetic field $H_\mathrm{app}$ (\textbf{case 1} from Section \ref{sec:problems}), we simulate the magnetic response of MNPs for a three-dimensional applied field, where no anisotropy is considered, i.e., $\varphi = 0$.
The applied field is then reconstructed from different data cases. The data consists either of the complete probability density function $u$ for the magnetization direction $m$, i.e., the observation operator $G= \mathrm{id}_{W(0,T)\to L^2((0,T)\times M)}$ is chosen, or it consists of the mean magnetic moment $\bar{m}$, i.e., $G = \mathrm{E}_{m\sim u}[m]$ is the expectation operator as in equation \eqref{eq:expectation}. We write $G = \mathrm{E}[m]$ for short. Both simulated measurements are equipped with Gaussian noise of different strengths. The approximate $L^2$-norm of the noise will be denoted as $\delta$ in each case, and the simulated noisy measurements are called $y^\delta$.

\begin{figure}
    \centering
    \includegraphics[width=\textwidth]{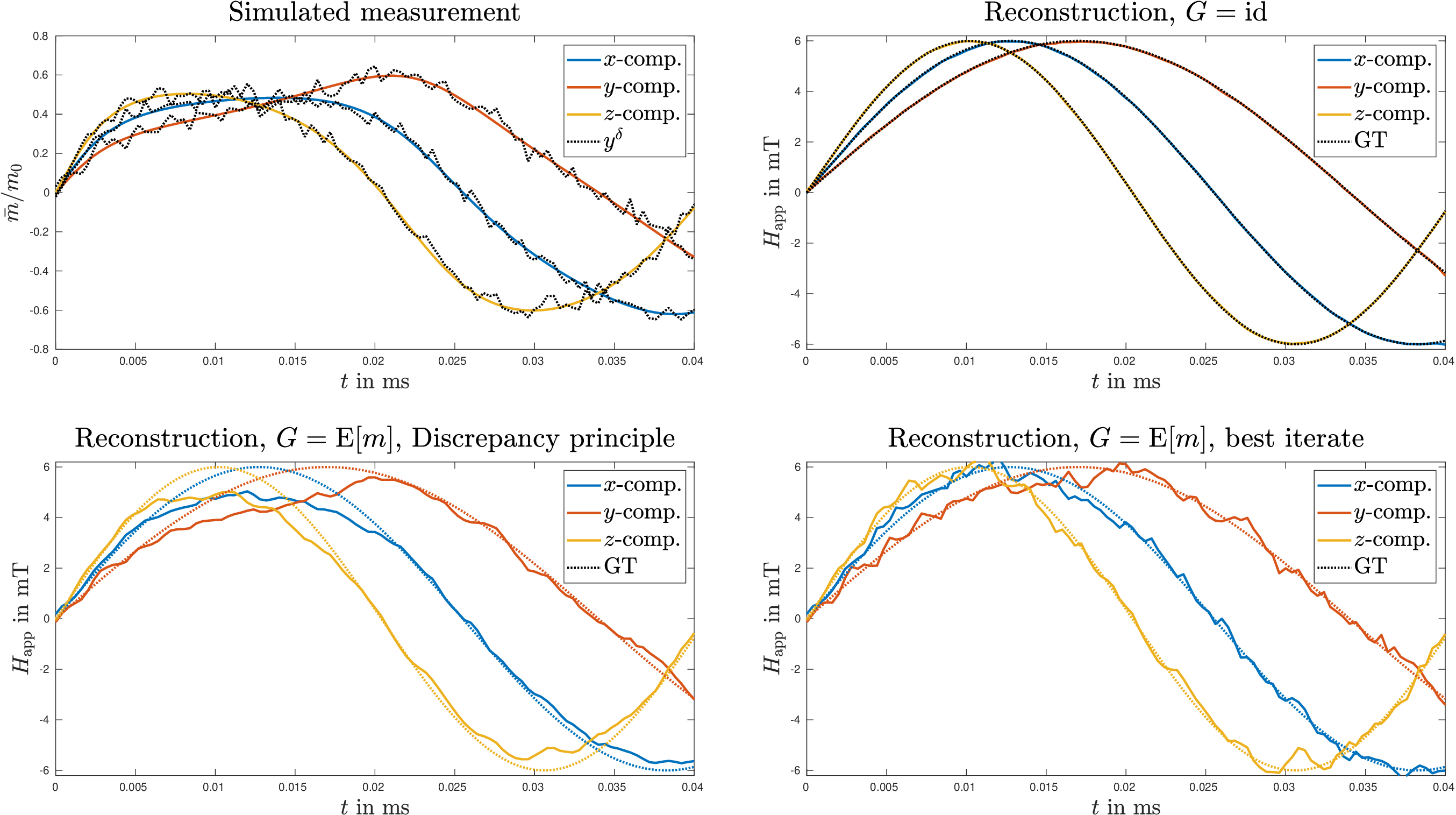}
    \caption{Reconstruction of the applied field $H_\mathrm{app}: [0,T] \to \mathbb{R}^3$ from a simulated measurement contaminated with 5\% Gaussian noise. Shown from the top left: The simulated measurement, the reconstruction without observation operator (i.e., the reconstruction is performed given the whole probability density function for $m$), with the expectation as the observation operator and the discrepancy principle and the Landweber iteration step which is closest to the ground truth (GT) in an $L^2$-sense, respectively.}
    \label{fig:Hnoise5}
\end{figure}

The results for 5\% noise are shown in Fig. \ref{fig:Hnoise5}. The simulated mean magnetic moment and corresponding Gaussian noise is shown as well as the reconstruction from the whole probability density function, from the mean using the discrepancy principle with $\tau = 1.1$, a choice that has been shown to work well for similar problems \cite{Hubmer_heuristic}, as well as the Landweber iterate with the lowest $L^2$-distance to the ground truth. This is shown in order to present the best theoretically achievable reconstruction for an optimal stopping criterion.
The results show that a good quality reconstruction is possible for the case of 5\% Gaussian noise. To obtain an accurate result, the discrepancy principle may not be sufficient. Instead, it might be necessary to apply further prior knowledge about $H_\mathrm{app}$ to the model or to use different stopping criteria, particularly if the noise level is not accurately known \cite{Hubmer_heuristic}.

Next, we consider \textbf{case 2} from Section \ref{sec:problems}. i.e., the reconstruction of the function $\varphi$. This can be understood as an unknown (effective) anisotropy landscape to be reconstructed from a measurement. For this, we simulate data for static uniaxial anisotropy in $y$-direction. The initial value is chosen to be uniaxial anisotropy in $x$-direction. Then, the anisotropy landscape is reconstructed. To simplify the problem, the anisotropy is assumed to be time-independent. For the magnetic field $H_\mathrm{app}$, a rapidly varying direction of the field vector is chosen such that the measurement includes the influence of the particle anisotropy in different directions.

\begin{figure}
    \centering
    \includegraphics[width=\textwidth]{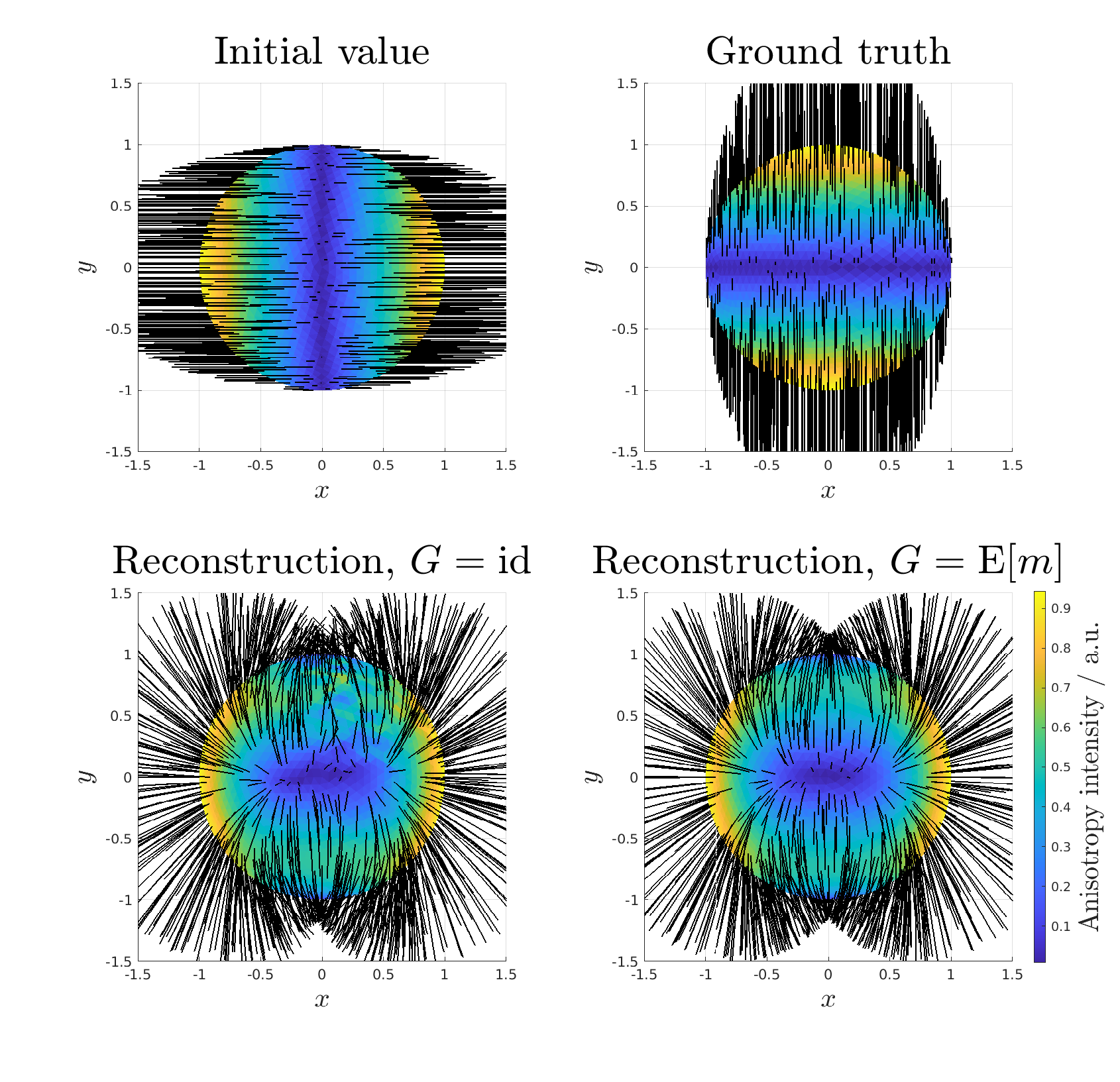}
    \caption{Reconstruction of a uniaxial anisotropy landscape. No noise is added. The sphere's surface color shows the anisotropy strength, while the black lines show its direction.}
    \label{fig:sd}
\end{figure}


The reconstruction results without noise and with and without an observation operator are presented in Fig. \ref{fig:sd}. In both cases, the discrepancy is minimized quite well, while the reconstructed anisotropy landscape differs substantially from the ground truth, even in the case $G= \mathrm{id}_{W(0,T)\to L^2((0,T)\times M)}$. This indicates that the sensitivity of the magnetic moments to changes in $\varphi$ may be low, so that further prior knowledge may be needed in order to obtain an accurate reconstruction for $\varphi$. However, the dependence on other factors such as the applied magnetic field and other particle parameters might have an influence on the performance.

For \textbf{case 3}, i.e., the reconstruction of a time-dependent easy axis $n(t)$, we simulate an ensemble of particles rotating  with constant velocity in the $xy$-plane. The applied magnetic field is a constant $\SI{10}{mT}$ in all directions, the anisotropy constant is chosen to be $\SI{1000}{\joule\per\cubic\meter}$.

\begin{figure}
    \centering
    \includegraphics[width=\textwidth]{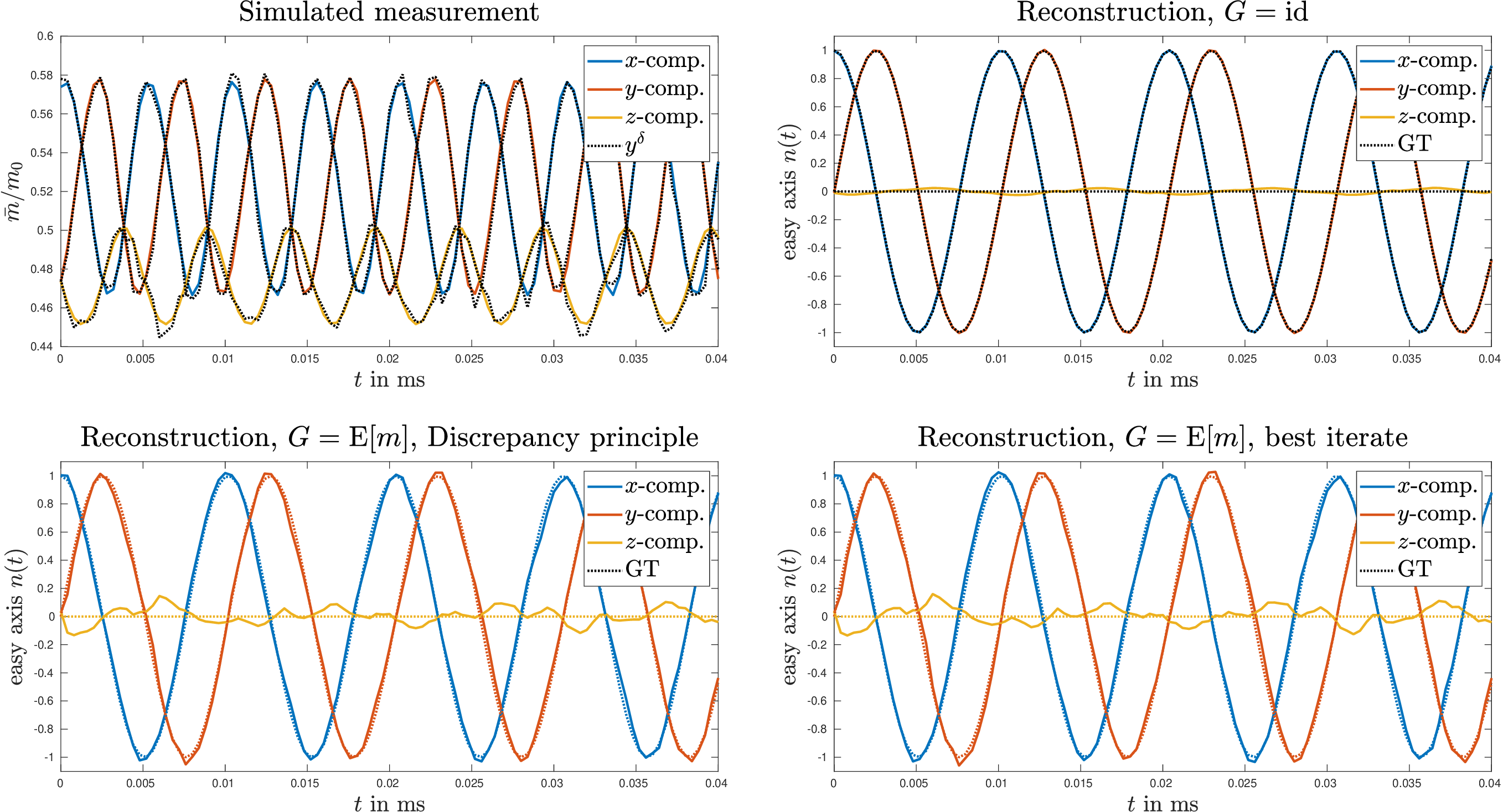}
    \caption{Reconstruction of the time-dependent easy axis $n(t)$ from a simulated measurement contaminated with 0.5\% Gaussian noise.}
    \label{fig:Nnoise05}
\end{figure}

\begin{figure}
    \centering
    \includegraphics[width=\textwidth]{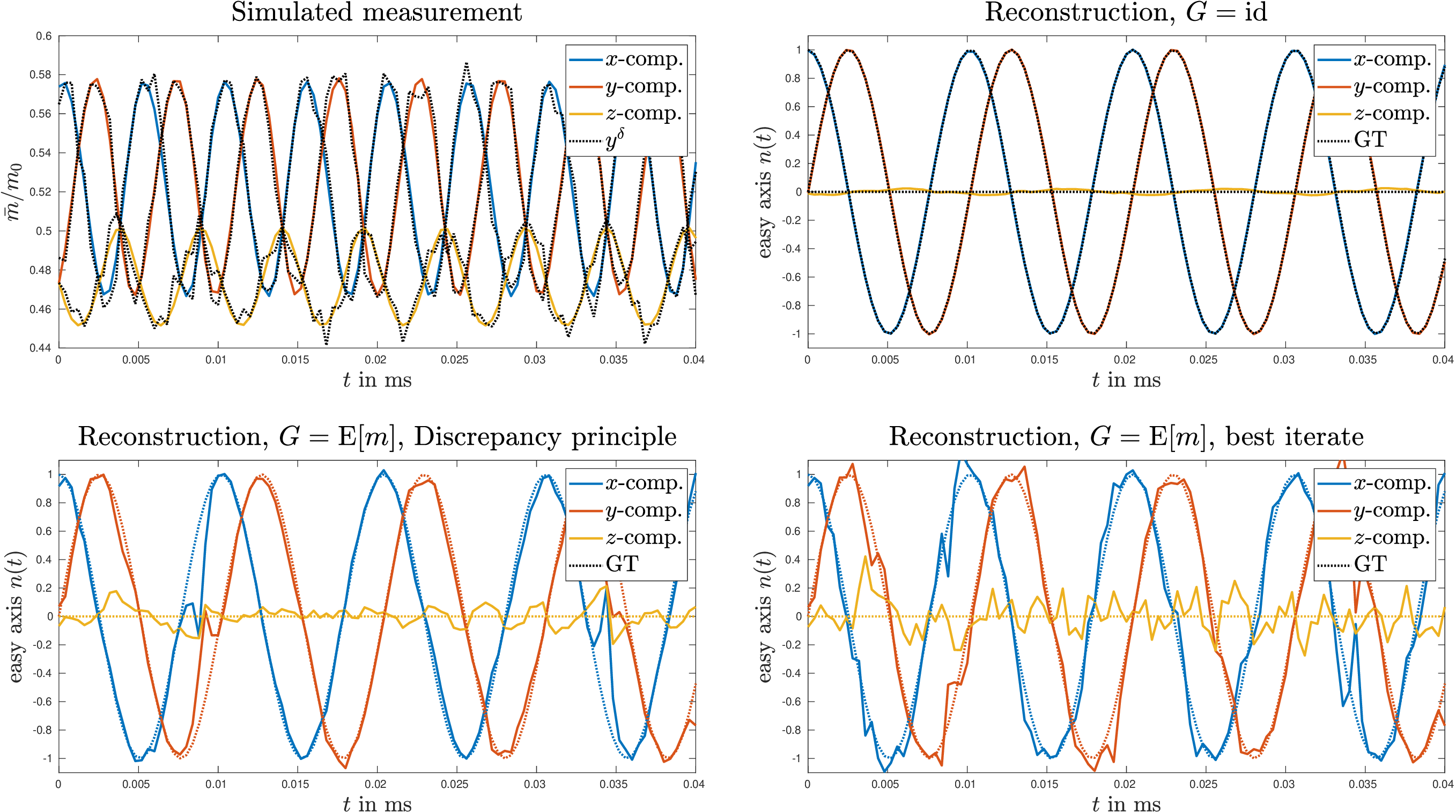}
    \caption{Reconstruction of $n(t)$ from simulations equipped with 1\% Gaussian noise.}
    \label{fig:Nnoise1}
\end{figure}

\begin{figure}
    \centering
    \includegraphics[width=\textwidth]{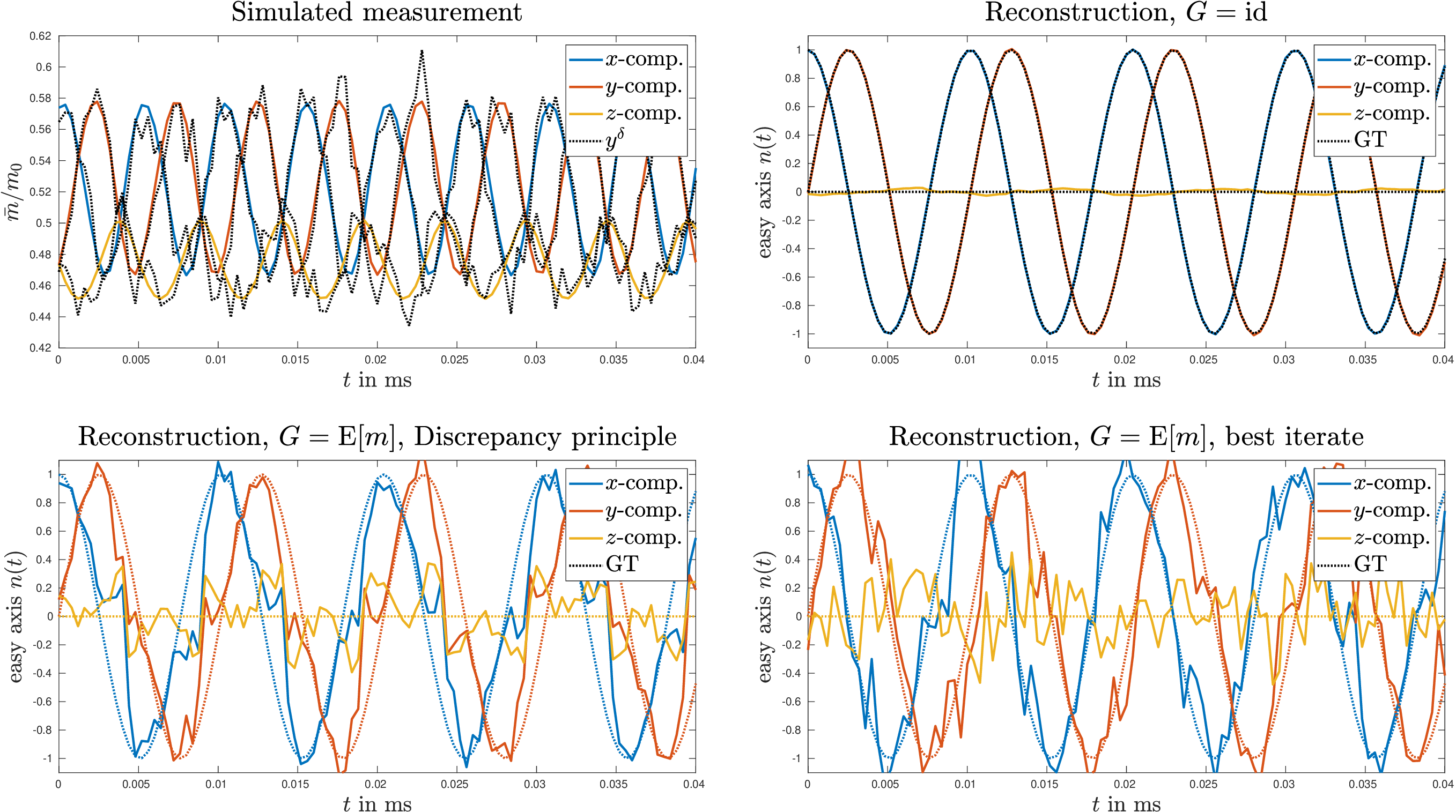}
    \caption{Reconstruction of $n(t)$ from simulations equipped with 2\% Gaussian noise. }
    \label{fig:Nnoise2}
\end{figure}

\begin{table}
\centering
\caption{Errors for the $n(t)$ reconstructions using the discrepancy principle.}
\label{table:n_errors}
\begin{tabular}{c | c c} 
  Noise level & $L^2$-error & $H^1$-error  \\ [0.2ex] 
 \hline
 0\% & 0.07 & 0.19  \\ 
 0.5\% & 0.08 & 0.2 \\
 1\% & 0.26 & 0.44 \\
 2\% & 0.56 & 0.78 \\
 5\% & 0.76 & 1.05 \\ [1ex] 
 
\end{tabular}

\end{table}

The results for 0.5\%, 1\%, and 2\% noise, respectively, are shown in Figs. \ref{fig:Nnoise05}-\ref{fig:Nnoise2}. Since the measurement amplitudes in this setting are much lower than for the previous one while the noise strength is expressed with respect to the measurement's $L^\infty$-norm, smaller noise levels have a greater impact on reconstruction quality. However, for the smaller noise regimes, the reconstruction is of high quality and the discrepancy principle is an effective stopping criterion. In Table \ref{table:n_errors}, the relative $L^2$-, as well as the  $H^1$-errors when using the discrepancy principle for different noise levels are listed. While in the noiseless case, complete accuracy is not achieved due to the numerical implementation and the presence of the observation operator, the error decreases with decreasing noise as expected.


\section{Discussion and Outlook}

In this work, we considered parameter identification problems for the Fokker-Planck equation on a compact manifold in the context of magnetic nanoparticle dynamics. Under reasonable assumptions, well-posedness of the forward equation, the existence of a minimizer for the discrepancy, and the Fr\'{e}chet differentiability of the forward operator have been proven. 

For different settings, namely reconstructing the magnetic field, reconstructing the particle anisotropy direction, and reconstructing the particle anisotropy landscape, numerical simulations have been performed. It was shown that the magnetic field and the (uniaxial) easy axis direction of the particles can be reconstructed with sufficient accuracy under reasonable noise conditions with the discrepancy principle. However, reconstructing the anisotropy landscape will require more prior information or different settings to lead to reasonable results. Additionally, if the noise level in the measurements is not accuately known in practice, heuristig stopping critera for the nonlinear Landweber method would have to be employed \cite{Hubmer_heuristic}.

On the theoretical side, while it was shown that a minimizer to the discrepancy functional exists, it remains an open question whether the Landweber method is guaranteed to converge to that optimum. For that, the tangential cone condition \cite{HankeNeubauerScherzer95} would have to be proven for this setting, as has been done for similar settings in \cite{Kaltenbacher_tangential_cone}. Also, while we employed the so-called reduced method for parameter identification in PDEs, which requires a solution of the forward equation in each step, the so-called all-at-once approach may introduce computational benefits by including the PDE in the discrepancy and optimizing for the data fit and the PDE simultaneously \cite{Kaltenbacher_2017}. In addition, different regularizing optimization techniques may achieve better results in future works, such as the Landweber-Kaczmarz method \cite{Nguyen_2019} or higher-order schemes such as the iteratively regularized Gauss-Newton method \cite{Kaltenbacher_iterative}.

On the applications side, an obvious next step would be to apply this technique to measured MPI signals. In this context, the unique noise distribution as well as the scanner receive chain, which results in band limitations, may pose challenges \cite{Paysen2020, Knopp2021efficient, Thieben2023}. Tracers used in MPI can exhibit a monodisperse \cite{tay2019pulsed} as well as a polydisperse behavior with respect to size and anisotropy distribution \cite{yoshida2013characterization, Albers_immobilized}. The proposed method is at this point only capable of identifying parameters in a monodisperse setting and would thus have to be extended or adjusted. 
An alternative field of potential applications is provided by immobilized and oriented nanoparticle markers \cite{MoeddelGrieseKluthKnopp2021} for example for the purpose of instrument orientation and position tracking which would require the joint identification of the marker position encoded in the applied field as well as the orientation given by the easy axis.  
In addition, the question if a collection of effective parameters could serve as a substitute for particle property distributions would be an interesting research question for future work.

Finally, repeated evaluations of the solution operator of the involved PDEs results in long computation times. In this direction, a replacement with learned methods approximating the forward operator could be a solution to shorten computation times. In \cite{Knopp2023fno}, it has been shown that such approaches are in principle feasible for rapidly obtaining solutions to the Fokker-Planck equation for nanoparticle dynamics with sufficient accuracy. Still, more research is needed in this area concerning the application to parameter identification.

\section*{Acknowledgment}
H. Albers and T. Kluth acknowledge funding by the German Research Foundation (DFG, Deutsche Forschungsgemeinschaft) - project
426078691.

\bibliographystyle{abbrv}
\bibliography{literature,literature-rsvd,literature-models}

\end{document}